\documentclass[11pt]{article}

\usepackage[T1]{fontenc}
\usepackage[utf8]{inputenc}
\usepackage{lmodern}
\usepackage{amsmath,amssymb,amsthm,mathtools}
\usepackage{stmaryrd}
\usepackage{enumitem}
\usepackage{microtype}
\usepackage[a4paper,margin=1.05in]{geometry}
\usepackage[colorlinks=true,linkcolor=blue,citecolor=blue,urlcolor=blue]{hyperref}

\newtheorem{theorem}{Theorem}[section]
\newtheorem{proposition}[theorem]{Proposition}
\newtheorem{lemma}[theorem]{Lemma}
\newtheorem{corollary}[theorem]{Corollary}

\newtheorem{example}[theorem]{Example}
\newtheorem{definition}[theorem]{Definition}
\newtheorem{question}[theorem]{Question}

\newcommand{\Fp}{\mathbb F_p}
\newcommand{\Zp}{\mathbb Z_p}
\newcommand{\cdp}{\operatorname{cd}}
\newcommand{\Homc}{\operatorname{Hom}_{\mathrm{cont}}}
\newcommand{\Tor}{\operatorname{Tor}}
\newcommand{\coker}{\operatorname{coker}}
\newcommand{\tdim}{\operatorname{tdim}}

\newcommand{\kappaGN}{\kappa(G,N)}

\title{Detecting Cohomological Dimension Three in Free-by-Demu\v{s}kin Pro-$p$ Groups}

\author{M. Palaisti}
\date{ }

\begin{document}
	
	\maketitle
	
	\begin{abstract}
Let
\[1\longrightarrow N\longrightarrow G\longrightarrow D\longrightarrow1\]
be exact, with $N\neq1$ free pro-$p$ of arbitrary rank and $D$ Demu\v{s}kin.
Writing $W=N/\Phi(N)$, we prove
\[H^3(G,\Fp)^\vee\simeq W^D,\qquad \cdp G=3\text{ if and only if }W^D\neq0.\]
Thus dimension three is detected on the first Frattini layer of the free
kernel. Finite-rank kernels always satisfy the criterion, so a dimension drop
requires an infinite-rank kernel, while the criterion remains stable under
passage to open subgroups. On the fixed-point-free locus $W^D=0$, the
remaining dimension-one/two alternative is governed by
$H^1(D,W^\vee)$ together with the Hochschild--Serre transgression. Dualizing
the latter produces a relation-defect class in the topological coinvariants
of $W$, which admits an explicit description in terms of a lifted defining
relation of the Demu\v{s}kin quotient. This yields a complete criterion for
the lower-dimensional branch and separates the top-degree fixed-point
obstruction from the extension data controlling dimensions one and two.
	\end{abstract}
	
	\medskip
	

	
\section{Introduction}
	
Let
\begin{equation}\label{eq:mainext}
1\longrightarrow N\longrightarrow G \overset{\pi}{\longrightarrow}D\longrightarrow1
\end{equation}
be an exact sequence of pro-$p$ groups. When $N$ is free pro-$p$ and $D$ is Demu\v{s}kin,
\[\cdp N=1,\qquad\cdp D=2,\]
and the standard extension inequality gives $\cdp G\le3$. We identify the kernel--quotient structure that determines whether this upper bound is attained.
	
Assume that $N\neq1$ is free pro-$p$, with no finiteness hypothesis on its rank, and put
\[W=N/\Phi(N).\]
Conjugation in $G$ induces a continuous $D$-action on the compact elementary abelian pro-$p$ group $W$, and Pontryagin duality identifies
\[W^\vee=\Homc(W,\Fp)=H^1(N,\Fp).\]
	
The main identity is
\begin{equation}
H^3(G,\Fp)^\vee\simeq \bigl(N/\Phi(N)\bigr)^D.
\end{equation}
Consequently,
\begin{equation}
\cdp G=3 \text{ if and only if } \bigl(N/\Phi(N)\bigr)^D\neq0.\end{equation}
Equivalently, dimension three occurs precisely when there exists
\[n\in N\setminus\Phi(N)\]
such that
\[[g,n]\in\Phi(N)\qquad \text{for every }g\in G.\]
	
The proof comes from the top edge of the Hochschild--Serre spectral sequence. Since $N$ has cohomological dimension one and $D$ has cohomological dimension two, the unique possible contribution to total degree three is $H^2\bigl(D,H^1(N,\Fp)\bigr)$. No differential enters or leaves this term. Top-degree Demu\v{s}kin duality then identifies its Pontryagin dual with $W^D$.
	
The fixed-Frattini criterion extends the finite-type dimension-additivity picture to arbitrary-rank free kernels. The duality-extension theorem of Schmidt--Wingberg \cite{SchmidtWingberg} gives dimension additivity when the kernel satisfies their finiteness condition $FC(N,p)$. For finitely generated free kernels, the criterion above recovers this conclusion directly from the induced action on $N/\Phi(N)$. For arbitrary-rank kernels, it continues to identify the top-degree obstruction even when the finite-type hypothesis is absent.
	
This distinction is realized by the examples of Kochloukova--Zalesskii \cite{KZ}. They construct a finitely generated free-by-Demu\v{s}kin pro-$p$ group of cohomological dimension two whose free kernel has infinite rank. The fixed-Frattini criterion gives
\[(N/\Phi(N))^D=0\]
for such an extension, thereby identifying the first-Frattini-layer mechanism underlying the dimension drop.
	
An infinite-rank kernel is therefore necessary, but not sufficient, for dimension drop. Indeed, if
\[G=N\times D,\]
then the induced action on $N/\Phi(N)$ is trivial and dimension three persists for free kernels of arbitrary nonzero rank. The relevant distinction is the presence or absence of a quotient-fixed direction in the first Frattini module.
	
When the fixed-point module is infinite, its ordinary Hamel dimension is not preserved by Pontryagin duality. For a compact elementary abelian pro-$p$ group $X$, write 
\[\tdim_{\Fp}X:=\dim_{\Fp}X^\vee,\]
and call this the \textit{topological $\Fp$-rank of $X$}. To measure the size of the fixed submodule of the first Frattini quotient, define the \textit{kernel fixed-point rank} by
\[\kappa(G,N) := \tdim_{\Fp} \bigl(N/\Phi(N)\bigr)^D.\]
Thus $\kappa(G,N)$ records the topological $\Fp$-rank of the $D$-fixed part of $N/\Phi(N)$. It moreover satisfies
\[\kappa(G,N)=\dim_{\Fp}H^3(G,\Fp)\]
for free kernels of arbitrary rank.
	
The lower-dimensional branch exhibits a complementary role for coinvariants. The top degree is controlled by $W^D$, whereas the Hochschild--Serre transgression is naturally dual to the topological coinvariant module
\[W_D= \displaystyle W/\overline{I_DW},\]
with $I_D$ the augmentation ideal of $\Fp\llbracket D\rrbracket$. On the fixed-point-free locus $W^D=0$, a relation-defect class $\delta_G\in W_D$ together with $H^1(D,W^\vee)$ determines whether the remaining cohomological dimension is one or two. Thus invariants govern the top-dimensional branch, while coinvariants enter through the transgression in the lower-dimensional branch.
	
The finite-rank dimension-three case also admits a duality refinement. Schmidt--Wingberg's theorem makes $G$ a duality group at $p$ of dimension 	three. Poincar\'e duality is more rigid. The normal-subgroup theorem of 	Grunewald--Jaikin-Zapirain--Pinto--Zalesskii \cite{GJPPZ} shows that, when $G$ is a pro-$p$ $PD^3$ group, a finitely generated free normal subgroup with Demu\v{s}kin quotient must be cyclic.
	
\section{Preliminaries}\label{sec:prelim}

We briefly recall the conventions and background used throughout the paper. All cohomology is continuous, and all module actions are understood to be continuous. Pontryagin duals of elementary abelian $p$-groups are taken with respect to $\Fp$. For general background on profinite and pro-$p$ groups and their cohomology, we refer to Serre \cite{Serre}, Ribes--Zalesskii \cite{RZ}, and Neukirch--Schmidt--Wingberg \cite{NSW}. For pseudocompact modules and completed group algebras, we use the formalism of Brumer \cite{Brumer}.
	
\subsection{Cohomological dimension}
	
For a pro-$p$ group $P$, the \textit{cohomological dimension} $\cdp P$ is the least integer $n$ such that
\[H^q(P,A)=0\]
for every discrete $p$-primary $P$-module $A$ and every $q>n$.
	
We use the standard criterion
\begin{equation}\label{eq:cdcriterion}
\cdp P\le n \text{ if and only if } H^{n+1}(P,\Fp)=0
\end{equation}
for pro-$p$ groups; see Serre \cite{Serre} or Neukirch--Schmidt--Wingberg \cite{NSW}.
	
A nontrivial pro-$p$ group has cohomological dimension one if and only if it is free pro-$p$. If
\[1\longrightarrow H\longrightarrow P\longrightarrow Q\longrightarrow1\]
is exact and both dimensions on the right are finite, the Hochschild--Serre spectral sequence gives
\begin{equation}\label{eq:upperbound}
\cdp P\le \cdp H+\cdp Q.
\end{equation}
	
\subsection{Demu\v{s}kin duality}
	
A \textit{Demu\v{s}kin group} $D$ is a finitely generated pro-$p$ group of cohomological dimension two such that
\[H^2(D,\Fp)\simeq\Fp\]
and the cup-product pairing on $H^1(D,\Fp)$ is nondegenerate. Equivalently, Demu\v{s}kin groups are the pro-$p$ Poincar\'e-duality groups of dimension two; see Labute \cite{Labute}, Serre \cite{Serre}, and \cite[Chapter~III]{NSW}.
	
The reduction modulo $p$ of the orientation character is trivial. Indeed, a continuous homomorphism from the pro-$p$ group $D$ to $\Fp^\times$ has trivial image.
	
\begin{lemma}\label{lem:topdual}
Let $D$ be Demu\v{s}kin and let $M$ be a discrete $\Fp[D]$-module. Then there is a natural isomorphism
\[H^2(D,M)^\vee \simeq \operatorname{Hom}_D(M,\Fp).\]
If $M=X^\vee$ for a compact $\Fp\llbracket D\rrbracket$-module $X$, then
\begin{equation}\label{eq:topdualfixed}
H^2(D,X^\vee)^\vee \simeq X^D.
\end{equation}\end{lemma}
	
\begin{proof}
For a finite discrete $M$, top-degree duality is the usual dualizing-module identity for a profinite duality group; see \cite[p.~1]{SchmidtWingberg}. Since $M$ is killed by $p$, the relevant homomorphisms into the dualizing module $I(D)$ take values in $I(D)[p]$. For a Demu\v{s}kin group this is a one-dimensional $\Fp$-space carrying the mod-$p$ orientation action. That action is trivial, and hence 
\[I(D)[p]\simeq\Fp\]
as a $D$-module. Therefore
\[H^2(D,M)^\vee\simeq \operatorname{Hom}_D(M,\Fp)\]
for finite $M$.
		
Now let $M$ be arbitrary and discrete. Every element of $M$ has open stabilizer and hence finite orbit, so $M$ is the filtered union of its finite-dimensional $D$-submodules:
\[M=\varinjlim_\lambda M_\lambda.\]
Continuous cohomology commutes with filtered direct limits of discrete coefficient modules; see \cite[Proposition~1.5.1]{NSW}. Thus
\[H^2(D,M)\simeq \varinjlim_\lambda H^2(D,M_\lambda).\]
Taking Pontryagin duals gives
\[H^2(D,M)^\vee\simeq \varprojlim_\lambda H^2(D,M_\lambda)^\vee \simeq \varprojlim_\lambda \operatorname{Hom}_D(M_\lambda,\Fp) \simeq \operatorname{Hom}_D(M,\Fp).\]
		
If $M=X^\vee$, Pontryagin biduality identifies $\operatorname{Hom}_D(X^\vee,\Fp)$ with the $D$-fixed subgroup of $X^{\vee\vee}\simeq X$, which proves \eqref{eq:topdualfixed}.
\end{proof}
	
\subsection{The Frattini module}

For a pro-$p$ group $N$, write
\[\Phi(N)=\overline{N^p[N,N]}\]
for its \textit{Frattini subgroup}. If $N$ is free pro-$p$, then
\[W=N/\Phi(N)\]
is a compact elementary abelian pro-$p$ group. We call $W$ the
\textit{Frattini module} of $N$. Its Pontryagin dual is naturally identified with
\[ H^1(N,\Fp)=W^\vee.\]
	
If $N\trianglelefteq G$ and $D=G/N$, conjugation induces a continuous $D$-action on $W$. Thus $W$ is a compact $\Fp\llbracket D\rrbracket$-module.
	
We use both the \textit{invariant module}
\[W^D= \{w\in W:dw=w\text{ for every }d\in D\}\]
and the \textit{topological coinvariant module}
\begin{equation}
W_D:=W/\overline{I_DW},
\end{equation}
where $I_D$ is the augmentation ideal of $\Fp\llbracket D\rrbracket$.
	
Pontryagin duality gives
\begin{equation}\label{eq:invcoinvdual}
(W_D)^\vee\simeq (W^\vee)^D.
\end{equation}
	
\begin{definition}\label{def:tdim}
For a compact elementary abelian pro-$p$ group $X$, define the topological $\mathbb{F}_p$-rank on $X$ as
\[\tdim_{\Fp}X:=\dim_{\Fp}X^\vee.\]
\end{definition}
	
When $X$ is finite, this agrees with its ordinary $\Fp$-dimension.

\section{Top-degree cohomology}\label{sec:top}

The purpose of this section is to isolate the top cohomological degree of a free-by-Demu\v{s}kin extension. The cohomological dimensions of the kernel and quotient force the Hochschild--Serre spectral sequence into a particularly simple range, so that the degree-three cohomology of $G$ can be read directly from a single term. Demu\v{s}kin duality then translates this term into a fixed-point condition on the first Frattini module of $N$.

Let \eqref{eq:mainext} be exact and put
\[V=H^1(N,\Fp)=W^\vee.\]
The Hochschild--Serre spectral sequence (see \cite[Theorem~2.4.3]{NSW} for the profinite setting) is
\begin{equation}
E_2^{i,j}=H^i\bigl(D,H^j(N,\Fp)\bigr)\Longrightarrow H^{i+j}(G,\Fp).
\end{equation}
	
Since $N$ is free,
\[H^j(N,\Fp)=0 \qquad (j\ge2),\]
and since $D$ is Demu\v{s}kin,
\[H^i(D,-)=0\qquad (i\ge3).\]
Thus the $E_2$-page has only two possibly nonzero rows and three possibly nonzero columns.

This restricted support separates the top-degree calculation from the
extension data that can affect lower degrees. In total degree three there is
only one possible contribution, while the transgression appears entirely in
the lower part of the spectral sequence.
	
The only potentially nonzero differential is the transgression
\begin{equation}\label{eq:transgression}
d_2^{0,1}: H^0(D,V)=V^D \longrightarrow	H^2(D,\Fp).
\end{equation}
It is the transgression appearing in the five-term sequence; see \cite[Proposition~1.6.7 and Theorem~2.4.3]{NSW}.

We first record the resulting top-edge identification. This is the point at which the cohomological-dimension assumptions on $N$ and $D$ do most of the work: no extension differential can alter the unique term of total degree three.
	
\begin{proposition}\label{prop:topedge}
There is a natural isomorphism
\begin{equation}\label{eq:edge}
H^3(G,\Fp) \simeq H^2(D,V).
\end{equation}
Moreover,
\[H^i(G,\Fp)=0 \qquad (i>3).\]
\end{proposition}
	
\begin{proof}
The unique $E_2$-term of total degree three is
\[E_2^{2,1}=H^2(D,V).\]
		
For $r\ge2$, a differential entering this term would have source $E_r^{2-r,r}$. Such a source either lies outside the first quadrant or has second index at least two, where the spectral sequence vanishes.
		
A differential leaving $E_r^{2,1}$ has target $E_r^{2+r,2-r}$. For $r=2$, the target lies in column four and vanishes because $\cdp D=2$; for $r>2$, the target lies outside the nonzero rows. Hence
\[E_\infty^{2,1}=E_2^{2,1}.\]
This is the only graded piece of the filtration of $H^3(G,\Fp)$, proving
\eqref{eq:edge}. The same support argument gives
\[H^i(G,\Fp)=0\qquad(i>3).\]
\end{proof}
	
The preceding proposition reduces the problem to the top cohomology of the Demu\v{s}kin quotient with coefficients in $V=W^\vee$. Applying Demu\v{s}kin duality now converts this cohomological description into a module-theoretic one.

\begin{theorem}\label{thm:topfixed}
For every extension \eqref{eq:mainext},
\begin{equation}
H^3(G,\Fp)^\vee\simeq W^D =\bigl(N/\Phi(N)\bigr)^D.
\end{equation}
Consequently,
\begin{equation}\label{eq:topdimension}
\dim_{\Fp}H^3(G,\Fp)=\tdim_{\Fp} \bigl(N/\Phi(N)\bigr)^D.
\end{equation}
\end{theorem}
	
\begin{proof}
Proposition~\ref{prop:topedge} gives
\[H^3(G,\Fp)\simeq H^2(D,W^\vee).\]
Lemma~\ref{lem:topdual} therefore yields
\[H^3(G,\Fp)^\vee\simeq H^2(D,W^\vee)^\vee \simeq W^D. \]
Equation \eqref{eq:topdimension} follows from Definition~\ref{def:tdim}.
	\end{proof}
	
Thus top-degree nonvanishing depends on the induced $D$-action on the first Frattini quotient $N/\Phi(N)$. Once this action is fixed, the transgression \eqref{eq:transgression} affects the lower cohomology but not $H^3(G,\Fp)$.

\section{Detecting cohomological dimension three}\label{sec:criterion}

The top-degree calculation yields an explicit criterion for whether the extension attains the maximal possible cohomological dimension. We now collect its equivalent cohomological, module-theoretic, and group-theoretic forms. In particular, the condition can be detected entirely on the first Frattini quotient of the free kernel.
	
\begin{theorem}\label{thm:equiv}
Let
\[1\longrightarrow N\longrightarrow G\longrightarrow D\longrightarrow1\]
be exact, with $N\neq1$ free pro-$p$ and $D$ Demu\v{s}kin. Then the following are equivalent:
\begin{enumerate}[label=\textup{(\roman*)}]
\item $\cdp G=3$;
\item $H^3(G,\Fp)\neq0$;
\item $H^2(D,H^1(N,\Fp))\neq0$;
\item $(N/\Phi(N))^D\neq0$;
\item there exists $n\in N\setminus\Phi(N)$ such that $[g,n]\in\Phi(N) \text{ for every }g\in G$.
\end{enumerate}\end{theorem}
	
\begin{proof}
By \eqref{eq:upperbound}, $\cdp G\le3$. Thus (i) and (ii) are equivalent by \eqref{eq:cdcriterion}. Proposition~\ref{prop:topedge} gives (ii) $\Longleftrightarrow$ (iii), and Theorem~\ref{thm:topfixed} gives (iii) $\Longleftrightarrow$ (iv). Finally, $n\Phi(N)$ is fixed by $D=G/N$ if and only if 
\[gng^{-1}n^{-1}\in\Phi(N)\qquad \text{for every }g\in G,\]
which is equivalent to
\[[g,n]\in\Phi(N).\]
The fixed class is nonzero precisely when $n\notin\Phi(N)$.
\end{proof}

The criterion records only whether the fixed submodule is nonzero. It is also useful to retain its size, which measures the dimension of the top cohomology of $G$.
	
\begin{definition}\label{def:kappa}
Let
\[ 1\longrightarrow N\longrightarrow G\longrightarrow D\longrightarrow1\]
be an exact sequence of pro-$p$ groups, where $N\neq1$ is free pro-$p$ and $D$ is Demu\v{s}kin. Writing 
\[W=N/\Phi(N),\]
define the kernel fixed-point rank of the extension by
\[\kappa(G,N):= \tdim_{\Fp} W^D =\tdim_{\Fp} \bigl(N/\Phi(N)\bigr)^D.\]
\end{definition}
	
\begin{corollary}\label{cor:kappa}
For every free-by-Demu\v{s}kin extension,
\[\kappa(G,N)=\dim_{\Fp}H^3(G,\Fp).	\]
In particular, $\cdp G=3 \text{ if and only if } \kappaGN>0$. \end{corollary}

The criterion is also compatible with passage to open subgroups. Thus the dimension-three property can be detected at every finite-index level by the corresponding Frattini module.

\begin{corollary}\label{cor:open}
Let $U\le G$ be open, put
\[N_U=N\cap U,\]
and let $D_U$ be the image of $U$ in $D$. Then
\[\cdp G=3 \iff \cdp U=3 \iff \bigl(N_U/\Phi(N_U)\bigr)^{D_U}\neq0.\]
	\end{corollary}
	
\begin{proof}
The group $N_U$ is open in $N$ and hence is again free pro-$p$; see, for example, \cite{RZ}. Since $N_U$ has finite index in the nontrivial torsion-free pro-$p$ group $N$, it is nontrivial. The image $D_U$ is open in $D$, hence is again Demu\v{s}kin.
		
Cohomological dimension is invariant under passage to open subgroups in this finite-cohomological-dimension setting, so $\cdp U=\cdp G$. Applying Theorem~\ref{thm:equiv} to
\[1\longrightarrow N_U\longrightarrow U\longrightarrow D_U\longrightarrow1\]
gives the result.
\end{proof}

\section{Finite-rank rigidity and the infinite-rank boundary} \label{sec:rank}

The fixed-point criterion becomes especially rigid when the free kernel has finite rank. In that case the Frattini module is finite-dimensional, and the pro-$p$ nature of the quotient forces every nonzero finite-dimensional $\Fp$-module to contain fixed vectors. This elementary observation rules out dimension drop in the finite-rank setting.
	
\begin{lemma}\label{lem:fixedfinite}
Let $P$ be a pro-$p$ group acting continuously on a nonzero finite-dimensional $\Fp$-vector space $M$. Then $M^P\neq0$.
	\end{lemma}
	
\begin{proof}
The action has finite image. Since $P$ is pro-$p$, its image is a finite $p$-group $\overline P$. The group algebra $\Fp[\overline P]$ is local and has the trivial module as its unique simple module. Every nonzero finite-dimensional module has nonzero socle and therefore contains a nonzero trivial submodule.
\end{proof}
	
Applying the lemma to the Frattini module gives the finite-rank consequence for free-by-Demu\v{s}kin extensions.

\begin{corollary}\label{cor:finite}
Suppose $N\neq1$ is finitely generated free pro-$p$. Then every extension
\[1\longrightarrow N\longrightarrow G\longrightarrow D\longrightarrow1\]
with $D$ Demu\v{s}kin satisfies $\cdp G=3$. Moreover,
\[\dim_{\Fp}H^3(G,\Fp)=\dim_{\Fp}(N/\Phi(N))^D\ge1.\]\end{corollary}
	
\begin{proof}
The Frattini quotient $	N/\Phi(N)$ is a nonzero finite-dimensional $\Fp$-vector space. Lemma~\ref{lem:fixedfinite} gives
\[(N/\Phi(N))^D\neq0,\]
and Theorem~\ref{thm:equiv} applies.
\end{proof}
	
For finitely generated kernels, Corollary~\ref{cor:finite} recovers the dimension-additivity consequence of Schmidt--Wingberg's duality-extension theorem directly from the induced action on the first Frattini quotient. The contrapositive is equally useful: any failure to attain cohomological dimension three must therefore occur beyond the finite-rank regime.

\begin{corollary}\label{cor:infiniterank}
If $\cdp G\le2$, then $N$ has infinite rank and $(N/\Phi(N))^D=0$.
\end{corollary}
	
This infinite-rank boundary is not merely formal. Known examples show that the dimension-drop branch is genuinely realized. Kochloukova--Zalesskii construct a finitely generated free-by-Demu\v{s}kin pro-$p$ group $G$ with infinite-rank free normal subgroup $N$ and $\cdp G=2$ \cite{KZ}. Theorem~\ref{thm:equiv} therefore gives 
\[(N/\Phi(N))^D=0.\]
This identifies the first-Frattini mechanism that accompanies their dimension drop.

Infinite rank alone, however, does not force a dimension drop. The simplest contrasting example is provided by the direct product.
	
\begin{example}[Infinite rank with dimension three]\label{ex:infinitecd3}
Let $N\neq1$ be free pro-$p$ of arbitrary rank and put
\[G=N\times D.\]
The $D$-action on $N/\Phi(N)$ is trivial, so
\[(N/\Phi(N))^D=N/\Phi(N)\neq0.\]
Hence $\cdp G=3$. Thus infinite rank allows both the dimension-three and dimension-drop branches.
\end{example}

The preceding example shows that infinite rank permits both sides of the fixed-point criterion. In fact, the possible fixed-point ranks can be realized completely.

\begin{theorem}\label{thm:kappa-realizability}
Let $D$ be an infinite Demu\v{s}kin group. Then every cardinal $\lambda$ occurs as $\kappa(G,N)$ for some free-by-Demu\v{s}kin extension
\[1\longrightarrow N \longrightarrow G \longrightarrow D\longrightarrow 1\]
with $N$ of infinite rank. More precisely, for every finite integer $k\ge 0$, there exists such an extension with $N$ free pro-$p$ of countably infinite rank and
\[\kappa(G,N)=k.\]
For every infinite cardinal $\lambda$, one may take $N$ to be free pro-$p$ of rank $\lambda$.
\end{theorem}

\begin{proof}
We first realize the finite values. Fix $k\ge 0$. By the classification of infinite Demu\v{s}kin groups, the abelianization of $D$ has positive $\Zp$-rank; see \cite{Labute} or \cite[Chapter~III, \S~9]{NSW}. Hence there is a continuous epimorphism
\[q:D \twoheadrightarrow \Zp.\]
	
Let $N$ be a free pro-$p$ group of countably infinite rank with a convergent free basis
\[x_1,x_2,x_3,\ldots \longrightarrow 1.\]
For each $n\ge1$, let
\[K_n=\overline{\langle x_{n+1},x_{n+2},\ldots\rangle}^{\,N}\]
be the closed normal subgroup generated by the tail beginning with $x_{n+1}$. Then
\[F_n:=N/K_n\]
is the free pro-$p$ group on the images of $x_1,\ldots,x_n$, and the natural quotient maps give \(N\simeq\varprojlim_n F_n\). 
	
For each $n$, define an automorphism $f_n\in\operatorname{Aut}(F_n)$ by
\[\begin{cases} 
f_n(x_j)=x_j&(j\le k),\\
f_n(x_j)=x_jx_{j+1}&(k<j<n),\; \; \text{ and }\\
f_n(x_n)=x_n&(n>k).
\end{cases}\]
If $n\le k$, this simply means that $f_n$ is the identity. Each $f_n$ is an automorphism. Indeed, it is triangular with respect to the free basis $x_1,\ldots,x_n$: beginning with $x_n$, which is fixed, one can recover successively $x_{n-1},x_{n-2},\ldots,x_{k+1}$ from their images. Equivalently, these are finite compositions of elementary Nielsen automorphisms.
	
The automorphisms $f_n$ are compatible with the natural projections $F_{n+1}\to F_n$. Hence they determine an automorphism
\[f\in\operatorname{Aut}(N)\]
satisfying
\[\begin{cases}
f(x_i)=x_i& (i\le k),\;\; \text{ and }\\
f(x_i)=x_ix_{i+1}&(i>k).
\end{cases}\]
	
We next show that the closed subgroup generated by $f$ is isomorphic to $\Zp$. Put
\[H=\overline{\langle f\rangle}\le\operatorname{Aut}(N).\]
For every $n$, let
\[H_n=\overline{\langle f_n\rangle}\le\operatorname{Aut}(F_n).\]
The action of $f_n$ on the Frattini quotient $\displaystyle F_n/\Phi(F_n)\simeq\Fp^n$ is represented by a unitriangular matrix. Its image therefore lies in a finite $p$-subgroup of $\operatorname{GL}_n(\Fp)$. The kernel of the natural homomorphism
\[\operatorname{Aut}(F_n)\longrightarrow\operatorname{GL}\bigl(F_n/\Phi(F_n)\bigr)\]
is pro-$p$, and consequently $H_n$ is a pro-$p$ group.
	
For $n\ge k+2$, the automorphism $f_n$ has infinite order. Indeed, on the abelianization $F_n^{\mathrm{ab}}\simeq\Zp^n$ we have
\[f_n(x_{n-1})=x_{n-1}+x_n,\;\; f_n(x_n)=x_n\]
in additive notation. Therefore, for every positive integer $m$,
\[f_n^m(x_{n-1})=x_{n-1}+mx_n.\]
Since $\Zp$ is torsion-free, this is different from $x_{n-1}$ for every $m>0$. Thus $f_n$ has infinite order. Since $H_n$ is an infinite procyclic pro-$p$ group, it follows that
\[H_n\simeq\Zp\qquad(n\ge k+2).\]
	
The actions on the quotients $F_n$ separate automorphisms of $N$. Hence the natural map
\[H\longrightarrow\prod_{n\ge1}H_n\]
is injective. Since each $H_n$ is pro-$p$, the group $H$ is pro-$p$. Moreover, the projection $H\to H_{k+2}$ is surjective, because the image contains the dense cyclic subgroup generated by $f_{k+2}$ and is compact, hence closed. Therefore $H$ is infinite. Since $H$ is by definition procyclic, we conclude that
\[H\simeq\Zp.\]
	
We now compute the induced action on the Frattini module $W=N/\Phi(N)$. Write
\[e_i=x_i\Phi(N).\]
Since the free basis of $N$ converges to $1$, the elements $e_i$ form a topological basis and $\displaystyle W\simeq\prod_{i\ge1}\Fp e_i$. Let $\gamma=f$ denote the distinguished topological generator of $H$, and put
\[\Lambda=\Fp\llbracket H\rrbracket\simeq\Fp\llbracket T\rrbracket,\qquad T=\gamma-1.\]
The action of $f$ on $W$ is
\[\begin{cases} 
\gamma e_i=e_i&(i\le k),\;\; \text{ and }\\
\gamma e_i=e_i+e_{i+1}&(i>k).\end{cases}\]
Hence
\[\begin{cases} 
Te_i=0&(i\le k),\;\; \text{ while }\\
Te_i=e_{i+1}&(i>k).\end{cases}\]
	
Let $\displaystyle W_1=\overline{\langle e_{k+1},e_{k+2},\ldots\rangle}$. Then
\[W=\Fp^k\times W_1,\]
where $H$ acts trivially on the first factor. Define
\[\theta:\Fp\llbracket T\rrbracket\longrightarrow W_1\]
by
\[\theta\left(\sum_{j\ge0}a_jT^j\right)=\sum_{j\ge0}a_je_{k+1+j}.\]
This map is continuous and bijective. Moreover,
\[\theta\left(T\sum_{j\ge0}a_jT^j\right)=T\theta\left(\sum_{j\ge0}a_jT^j\right),\]
because $Te_{k+1+j}=e_{k+2+j}$. Thus $\theta$ is an isomorphism of compact $\Fp\llbracket H\rrbracket$-modules. Consequently,
\[W\simeq\Fp^k\times\Fp\llbracket H\rrbracket\simeq\Fp^k\times\Fp\llbracket T\rrbracket.\]
	
The regular module $\Fp\llbracket T\rrbracket$ has no nonzero $H$-fixed element. Indeed, an element $u\in\Fp\llbracket T\rrbracket$ is fixed by $H$ if and only if it is fixed by $\gamma$, equivalently \(Tu =0\). Since $\Fp\llbracket T\rrbracket$ is an integral domain, multiplication by $T$ is injective, and therefore $u=0$. Hence
\[W^H\simeq\Fp^k.\]
	
Choose an isomorphism \(\Zp\overset{\sim}{\longrightarrow}H\) sending $1$ to $f$, let $D$ act on $N$ through the composite
\[D\overset{q}{\twoheadrightarrow}\Zp\overset{\sim}{\longrightarrow}H\le\operatorname{Aut}(N),\]
and form the semidirect product
\[G=N\rtimes D.\]
Then
\[1\longrightarrow N\longrightarrow G\longrightarrow D\longrightarrow1\]
is a free-by-Demu\v{s}kin extension. Since the action of $D$ on $W$ factors surjectively through $H$,
\[W^D=W^H\simeq\Fp^k.\]
Therefore
\[\kappa(G,N)=\tdim_{\Fp}W^D=k.\]
This realizes every finite value $k\ge0$ with $N$ of countably infinite rank.
	
It remains to realize the infinite cardinal values. Let $\lambda$ be an infinite cardinal, let $N_\lambda$ be a free pro-$p$ group of rank $\lambda$ and let $D$ act trivially on $N_\lambda$, and put
\[G_\lambda=N_\lambda\times D.\]
Then
\[\bigl(N_\lambda/\Phi(N_\lambda)\bigr)^D=N_\lambda/\Phi(N_\lambda).\]
By the definition of the rank of a free pro-$p$ group, $\dim_{\Fp}H^1(N_\lambda,\Fp)=\lambda$. Since
\[H^1(N_\lambda,\Fp)=\bigl(N_\lambda/\Phi(N_\lambda)\bigr)^\vee,\]
Definition~\ref{def:tdim} gives $\tdim_{\Fp}\bigl(N_\lambda/\Phi(N_\lambda)\bigr)=\lambda$. Hence
\[\kappa(G_\lambda,N_\lambda)=\lambda.\]
	
Thus every cardinal occurs as the kernel fixed-point rank of a free-by-Demu\v{s}kin extension with free kernel of infinite rank, and every finite value already occurs with a countably infinite-rank kernel.
\end{proof}

\section{Duality inside the dimension-three branch} \label{sec:duality}
	
The fixed-Frattini criterion detects cohomological dimension. For finite-rank kernels, classical extension theory gives the stronger duality structure.
	
\subsection{Finite-rank kernels}

The point is that dimension three by itself is weaker than duality. For finitely generated free kernels, the finiteness hypotheses required by the duality-extension theorem are automatically satisfied, so the dimension-three conclusion strengthens to a duality statement.

We first record this finite-rank refinement.

\begin{proposition}\label{prop:duality}
If $N\neq1$ is finitely generated free pro-$p$ and $D$ is Demu\v{s}kin, then $G$ is a profinite duality group at $p$ of dimension three.
\end{proposition}
	
\begin{proof}
Schmidt--Wingberg's condition $FC(N,p)$ requires $H^i(U,\Fp)$ to be finite for every open subgroup $U\le N$ and every $i\ge0$; see \cite[p.~2]{SchmidtWingberg}. If $N$ is finitely generated free pro-$p$, then every open subgroup $U$ is again finitely generated free pro-$p$. Therefore
\[H^i(U,\Fp)=0\qquad(i\ge2),\]
while $H^0(U,\Fp)\simeq\Fp$ and $H^1(U,\Fp)$ is finite-dimensional. Hence $FC(N,p)$ holds.
		
The group $N$ is a duality group at $p$ of dimension one, and the Demu\v{s}kin group $D$ is a duality group at $p$ of dimension two. Schmidt--Wingberg's extension theorem \cite[Theorem~1(ii)]{SchmidtWingberg} therefore gives that $G$ is a duality group at $p$, with
\[\cdp G=\cdp N+\cdp D= 1+2=3.\]
\end{proof}
	
\subsection{The $PD^3$ case}
	
Poincar\'e duality imposes a substantially stronger condition than profinite
duality. In the present setting, this additional rigidity is reflected in the
structure of the free normal subgroup. Poincar\'e-duality extensions were studied by Wingberg \cite{WingbergPD}. In dimension three, normal-subgroup structure gives a strong restriction on free kernels. 

The normal-subgroup theorem for pro-$p$ $PD^3$ groups then forces the free kernel to have rank one.
	
\begin{proposition}	\label{prop:PDnecessary}
Assume that $N\neq1$ is finitely generated free pro-$p$, $D$ is Demu\v{s}kin, and $G$ is a pro-$p$ $PD^3$ group. Then $N\simeq\Zp$.\end{proposition}
	
\begin{proof}
By \cite[Corollary~1.5]{GJPPZ}, a finitely generated normal subgroup of a pro-$p$ $PD^3$ group with infinite quotient is either infinite cyclic with virtually Demu\v{s}kin quotient, or Demu\v{s}kin with virtually infinite-cyclic quotient.
		
Here $N$ is free pro-$p$ and $G/N=D$ is Demu\v{s}kin. The second alternative is impossible because a nontrivial free pro-$p$ group has cohomological dimension one. Hence $N\simeq\Zp$. 
\end{proof}

These results place the fixed-Frattini criterion within a sequence of
increasingly restrictive structural conditions.

Thus the dimension-three branch has the hierarchy
\[ (N/\Phi(N))^D\neq0 \quad\Longleftrightarrow\quad \cdp G=3, \]
\[N\ \text{finitely generated} \quad\Longrightarrow\quad G\ \text{is a duality group of dimension }3, \]
and, under the finite-rank hypothesis,
\[ G\ \text{a }PD^3\text{ group }\quad\Longrightarrow\quad N\simeq\Zp.
\]

\section{The lower-dimensional complement}
\label{sec:complement}

The preceding sections identify the dimension-three branch through the nonvanishing of the fixed submodule $W^D$. We now turn to the complementary case, where this top-degree contribution disappears. The remaining problem is to distinguish whether the extension has cohomological dimension one or two.
	
Assume now $W^D=0$. By Theorem~\ref{thm:equiv},
\[H^3(G,\Fp)=0 \qquad\text{and}\qquad \cdp G\le2. \]
	
Let
\[\tau=d_2^{0,1}:V^D \longrightarrow H^2(D,\Fp) \simeq \Fp\]
be the Hochschild--Serre transgression. The $E_\infty$ filtration gives the exact sequences
\begin{equation}
0\longrightarrow H^1(D,\Fp) \longrightarrow H^1(G,\Fp) \longrightarrow \ker\tau \longrightarrow0
\end{equation}
and
\begin{equation}\label{eq:H2low}
0 \longrightarrow \coker\tau \longrightarrow H^2(G,\Fp) \longrightarrow H^1(D,V) \longrightarrow0.
	\end{equation}

Thus, once the degree-three contribution vanishes, the remaining obstruction is concentrated in two pieces: the group $H^1(D,V)$ and the transgression from $V^D$ to the one-dimensional space $H^2(D,\Fp)$.
	
\subsection{The relation defect}

The transgression admits a convenient dual interpretation on the coinvariant side of the Frattini module. This leads to a single class in $W_D$ that records whether the transgression is nonzero.
	
By \eqref{eq:invcoinvdual},
\[V^D= (W^\vee)^D \simeq (W_D)^\vee. \]
Since
\[H^2(D,\Fp)\simeq\Fp,\]
the dual transgression is a map
\[\tau^\vee: H^2(D,\Fp)^\vee \longrightarrow W_D. \]
	
\begin{definition}\label{def:defect}
Let
\[1\longrightarrow N\longrightarrow G\longrightarrow D\longrightarrow1\]
be as above, let
\[\tau:H^1(N,\Fp)^D\longrightarrow H^2(D,\Fp)\]
be the Hochschild--Serre transgression, and let
\[W_D=W/\overline{I_DW}.\]
Choose a nonzero element $\eta^\vee\in H^2(D,\Fp)^\vee$ and define
\[\delta_G:= \tau^\vee(\eta^\vee) \in W_D. \]
The class $\delta_G$ is well defined up to multiplication by a nonzero scalar, so the condition $\delta_G\neq0$ is intrinsic.
\end{definition}

The abstractly defined class $\delta_G$ also has a concrete interpretation
in terms of a defining relation of the Demu\v{s}kin quotient.

\begin{lemma}
\label{lem:relationdefect}
Let
\[1\longrightarrow R\longrightarrow F\longrightarrow D\longrightarrow1\]
be a minimal pro-$p$ presentation of $D$, with $F$ free pro-$p$, and write
\[R=\overline{\langle\!\langle r\rangle\!\rangle}\]
for a defining Demu\v{s}kin relation. Choose lifts in $G$ of a free basis of $F$. They determine a homomorphism
\[\alpha:F\longrightarrow G\]
lifting the identity on $D$. Put
\[w:=\alpha(r)\in N.\]
Then, up to the scalar normalization in Definition~\ref{def:defect},
\[\delta_G=\overline{w}\in W_D. \]
	\end{lemma}
	
\begin{proof}
Because the presentation of $D$ is minimal, inflation induces an isomorphism
\[H^1(D,\Fp)\overset{\sim}{\longrightarrow}H^1(F,\Fp). \]
Since $F$ is free pro-$p$, $H^2(F,\Fp)=0$. The Hochschild--Serre five-term sequence \cite[Proposition~1.6.7]{NSW} therefore identifies the presentation transgression
\[\operatorname{trg}_R: H^1(R,\Fp)^D \overset{\sim}{\longrightarrow} H^2(D,\Fp).\]
		
The standard relation/transgression pairing is
\[ \frac{R}{R^p[R,F]} \times H^1(R,\Fp)^D \longrightarrow \Fp, \qquad
(\overline{s},\psi) \longmapsto \psi(s). \]
The canonical identification
\[H^1(R,\Fp)^D\simeq \operatorname{Hom}_{\Fp} \left(R/R^p[R,F],\Fp\right)\]
gives the perfect evaluation pairing
\[\frac{R}{R^p[R,F]} \times H^1(R,\Fp)^D \longrightarrow \Fp, \qquad (\overline{s},\psi) \longmapsto \psi(s). \]
For a minimal pro-$p$ presentation, the images of a minimal system of normal generators of $R$ form an $\Fp$-basis of $R/R^p[R,F]$; see \cite[Corollary~3.9.3]{NSW}. Since $D$ is one-relator and 
\[R=\overline{\langle\!\langle r\rangle\!\rangle},\]
the class of $r$ is a basis of this relation module.

The chosen lifts give a morphism of extensions
\[\begin{array}{ccccccccc}
1&\longrightarrow&R&\longrightarrow&F&\longrightarrow&D&\longrightarrow&1\\
&&\downarrow&&\downarrow\alpha&&\Vert\\
1&\longrightarrow&N&\longrightarrow&G&\longrightarrow&D&\longrightarrow&1
\end{array}
\]
Naturality of the five-term sequence gives the commutative square
\[\begin{array}{ccc}
H^1(N,\Fp)^D & \overset{\tau}{\longrightarrow} & H^2(D,\Fp)\\
\downarrow && \Vert\\
H^1(R,\Fp)^D & \overset{\operatorname{trg}_R}{\longrightarrow} & H^2(D,\Fp),
\end{array}\]
where the left vertical map is restriction along $R\longrightarrow N$. Let
\[\varphi\in H^1(N,\Fp)^D.\] Under the relation pairing, the presentation transgression is evaluated on the defining relation by
\[(\varphi\circ\alpha)(r)= \varphi(\alpha(r)) = \varphi(w). \]
Hence the functional on $V^D$ defined by the transgression is evaluation at the image of $w$.
		
Every $D$-invariant continuous functional on $W$ vanishes on $\displaystyle \overline{I_DW}$ and therefore factors through $W_D$. Under
\[(V^D)^\vee\simeq W_D,\]
the dual transgression is consequently represented by the class of $w$ in $W_D$.
\end{proof}
	
\subsection{Dimension one versus dimension two}

We can now combine the cohomological term $H^1(D,W^\vee)$ with the relation defect to obtain an exact criterion for the lower-dimensional branch.
	
\begin{theorem}\label{thm:lower}
Assume $W^D=0$. Then $\cdp G=1$ if and only if
\begin{equation}\label{eq:lowerconditions}
H^1(D,W^\vee)=0 \qquad\text{and}\qquad \delta_G\neq0.
\end{equation}
In every other case, $\cdp G=2$. 
\end{theorem}
	
\begin{proof}
Since $W^D=0$, Theorem~\ref{thm:equiv} gives $\cdp G\le2$. From \eqref{eq:H2low}, $H^2(G,\Fp)=0$ if and only if $H^1(D,V)=0$ and $\coker\tau=0$. The target of
\[\tau:V^D\longrightarrow H^2(D,\Fp)\]
is one-dimensional. Therefore
\[\coker\tau=0 \iff \tau\neq0 \iff \tau^\vee\neq0 \iff \delta_G\neq0.\]
Since $V=W^\vee$, we obtain
\[H^2(G,\Fp)=0 \text{ if and only if } H^1(D,W^\vee)=0\text{ and } \delta_G\neq0.\]
		
A nontrivial pro-$p$ group has cohomological dimension one precisely when $H^2(G,\Fp)=0$. If $H^2(G,\Fp)\neq0$, the inequality $\cdp G\le2$ gives $\cdp G=2$.
\end{proof}
	
The vanishing of $H^1(D,W^\vee)$ can in turn be guaranteed by a standard
module-theoretic hypothesis.
	
\begin{proposition}\label{prop:projectsuff}
Let $\Lambda=\Fp\llbracket D\rrbracket$. If $W$ is projective as a pseudocompact left $\Lambda$-module, then
\[H^1(D,W^\vee)=0.\]
Consequently, if
\[W^D=0, \qquad W\text{ is projective}, \qquad \delta_G\neq0, \]
then $\cdp G=1$. 
\end{proposition}
	
\begin{proof}
Pontryagin duality between continuous cohomology and profinite homology gives
\[H^1(D,W^\vee)^\vee \simeq H_1(D,W). \]
In the pseudocompact module category,
\[H_1(D,W) \simeq \Tor_1^\Lambda(\Fp,W). \]
Projectivity of $W$ makes this Tor group vanish, and hence $H^1(D,W^\vee)=0$. The final assertion follows from Theorem~\ref{thm:lower}.
	\end{proof}

Taken together, the top- and lower-degree analyses give a two-stage
description of the cohomological dimension of the extension. The two branches are therefore described by
\[W^D\neq0 \quad\Longleftrightarrow\quad \cdp G=3, \]
while on $W^D=0$ the pair
\[H^1(D,W^\vee), \qquad \delta_G\in W_D \]
distinguishes dimensions one and two.

\section{Examples}\label{sec:examples}
	
\begin{example}[Direct products]
Let $N\neq1$ be free pro-$p$ of arbitrary rank and put
\[G=N\times D.\]
The $D$-action on
\[W=N/\Phi(N)\]
is trivial. Hence $W^D=W\neq0$ and $\cdp G=3$. Moreover,
\[\kappa(G,N)=\tdim_{\Fp}W=\dim_{\Fp}H^1(N,\Fp).\]
\end{example}
	
\begin{example}[Finite-rank nontrivial actions]
Let $N$ be free pro-$p$ of finite positive rank and let $D$ act continuously on $N$. For the corresponding semidirect product 
\[G=N\rtimes D,\]
Lemma~\ref{lem:fixedfinite} gives
\[(N/\Phi(N))^D\neq0.\]
Thus $\cdp G=3$ for every such action.
\end{example}
	
\begin{example}[The Kochloukova--Zalesskii phenomenon]
Kochloukova--Zalesskii construct a finitely generated free-by-Demu\v{s}kin pro-$p$ group with infinite-rank free kernel and $\cdp G=2$ \cite{KZ}. Theorem~\ref{thm:equiv} gives
\[(N/\Phi(N))^D=0.\]
Thus their example lies on the fixed-point-free locus of the first Frattini module, and the theorem identifies the first-layer mechanism responsible for the loss of the top cohomological degree. 
\end{example}

\begin{example}[A free middle group]\label{ex:freemiddle}
Let
\[1\longrightarrow N \longrightarrow F \longrightarrow D \longrightarrow1 \]
be a minimal one-relator presentation of the Demu\v{s}kin group $D$, with $F$ free pro-$p$ and
\[N=\overline{\langle\!\langle r\rangle\!\rangle}.\]
Since $\cdp F=1$, Theorem~\ref{thm:equiv} gives $(N/\Phi(N))^D=0$. Theorem~\ref{thm:lower} then gives $	H^1\bigl(D,(N/\Phi(N))^\vee\bigr)=0$ and $\delta_F\neq0$. By Lemma~\ref{lem:relationdefect}, $\delta_F$ is represented in the topological coinvariants by the defining relation $r$.
\end{example}
	
\section{Relation to existing results}\label{sec:scope}
	
Schmidt--Wingberg's extension theorem \cite{SchmidtWingberg} places the finite-rank free-by-Demu\v{s}kin case within profinite duality theory. For a finitely generated free kernel, their finiteness condition holds and the middle group is a duality group of dimension three. Theorem~\ref{thm:equiv} gives the corresponding dimension statement directly from the first Frattini quotient and extends the same fixed-point detector to arbitrary-rank kernels.
	
Kochloukova--Zalesskii \cite{KZ} exhibit the complementary infinite-rank phenomenon: a finitely generated middle group of cohomological dimension two can have an infinite-rank free kernel and a Demu\v{s}kin quotient. The fixed-Frattini criterion identifies the associated obstruction as
\[(N/\Phi(N))^D=0.\]
Together with Example~\ref{ex:infinitecd3}, this shows that arbitrary rank permits both sides of the criterion and that the first-layer action, rather than rank alone, distinguishes them.
	
For $PD^3$ middle groups, normal-subgroup theory imposes stronger restrictions. The theorem of Grunewald--Jaikin-Zapirain--Pinto--Zalesskii \cite{GJPPZ} gives the rank-one conclusion in Proposition~\ref{prop:PDnecessary}, while Castellano--Zalesskii \cite{CastellanoZalesskii} obtain further structural results for subgroups of pro-$p$ $PD^3$ groups.
	
The same top-degree mechanism extends from Demu\v{s}kin quotients to pro-$p$ Poincar\'e-duality quotients of arbitrary dimension. 

\begin{proposition}\label{prop:PDn}
Let
\[1\longrightarrow N\longrightarrow G\longrightarrow Q\longrightarrow1\]
be exact, where $N\neq1$ is free pro-$p$ and $Q$ is a pro-$p$ Poincar\'e-duality group of dimension $n$. Then
\[H^{n+1}(G,\Fp)^\vee \simeq \bigl(N/\Phi(N)\bigr)^Q. \]
Consequently, $\cdp G=n+1$ if and only if $\bigl(N/\Phi(N)\bigr)^Q\neq0$. \end{proposition}
	
\begin{proof}
Since $\cdp N=1$ and $\cdp Q=n$, we have that $\cdp G\le n+1$. In the Hochschild--Serre spectral sequence, the only possibly nonzero term of total degree $n+1$ is 
\[E_2^{n,1}=H^n\bigl(Q,H^1(N,\Fp)\bigr).\]
As in Proposition~\ref{prop:topedge}, no differential can enter or leave this term. Hence
\[H^{n+1}(G,\Fp)\simeq H^n\bigl(Q,H^1(N,\Fp)\bigr).\]
		
Put \[W=N/\Phi(N).\]
Top-degree $PD^n$ duality gives
\[H^n(Q,W^\vee)^\vee \simeq \operatorname{Hom}_Q(W^\vee,I(Q)[p]),\]
where $I(Q)$ is the dualizing module. The same passage from finite to arbitrary discrete coefficients used in Lemma~\ref{lem:topdual} applies here, using \cite[Proposition~1.5.1]{NSW}.
		
The mod-$p$ orientation action is trivial because $Q$ is pro-$p$, so
\[I(Q)[p]\simeq\Fp\]
as a $Q$-module. Therefore
\[H^{n+1}(G,\Fp)^\vee \simeq \operatorname{Hom}_Q(W^\vee,\Fp) \simeq W^Q.\]
The cohomological-dimension criterion completes the proof.
	\end{proof}
	
For $n=2$, the quotient has one-dimensional top cohomology and the two-row spectral sequence also yields the relation-defect description of the lower-dimensional branch developed in Section~\ref{sec:complement}.
	
\section{Further questions}

The fixed-Frattini formulation leaves several natural realizability problems. The preceding results separate the module-theoretic datum $W=N/\Phi(N)$ from the additional extension data recorded by the relation defect.

If the extension
\[1\longrightarrow N\longrightarrow G\longrightarrow D\longrightarrow1\]
splits, then $\delta_G=0$. Indeed, writing $G=N\rtimes D$, every element of $H^1(N,\Fp)^D$ extends to a homomorphism $G\longrightarrow\Fp$ by setting it equal to zero on $D$. Hence the restriction map
\[H^1(G,\Fp)\longrightarrow H^1(N,\Fp)^D\]
is surjective. By the Hochschild--Serre five-term sequence, the transgression
\[\tau:H^1(N,\Fp)^D\longrightarrow H^2(D,\Fp)\]
is therefore zero, and thus $\delta_G=0$.

\begin{question}
Let $D$ be a fixed Demu\v{s}kin group. Which compact $\Fp\llbracket D\rrbracket$-modules $W$ with $W^D=0$ occur as $W\simeq N/\Phi(N)$ for an infinite-rank free normal subgroup $N$ in a free-by-Demu\v{s}kin extension? Equivalently, which such $D$-actions lift to outer actions on free pro-$p$ groups, and in the split case to homomorphisms
\[D\longrightarrow\operatorname{Aut}(N)?\]
	
For a realizable $W$, which relation defects can occur? In particular, when can $\delta_G$ vanish, which classes $[\delta_G]\in\mathbb P(W_D)$ can arise for a fixed $D$-module structure on $W$, and how is this related to $H^1(D,W^\vee)$?
\end{question}

\begin{question}
Suppose that $G$ occurs as a maximal pro-$p$ Galois group, or as a naturally occurring closed subgroup of one. What restrictions does Galois theory impose on the $D$-module $N/\Phi(N)$, including its fixed submodule and relation defect, and can these restrictions rule out particular free-by-Demu\v{s}kin extensions?
\end{question}

\section*{Acknowledgments} 

Theorem \ref{thm:kappa-realizability} appeared in an earlier version of this manuscript as Question 10.1. We are grateful to Henrique Souza, who privately communicated an argument essentially resolving the question and forming the basis of the theorem and proof presented here.


\begin{thebibliography}{99}
			
			\bibitem{Brumer}
			A. Brumer,
			\textit{Pseudocompact algebras, profinite groups and class formations},
			J. Algebra \textbf{4} (1966), no.~3, 442--470.
			\href{https://doi.org/10.1016/0021-8693(66)90034-2}
			{doi:10.1016/0021-8693(66)90034-2}.
			
			\bibitem{CastellanoZalesskii}
			I. Castellano and P. Zalesskii,
			\textit{Subgroups of pro-$p$ $PD^3$-groups},
			Monatsh. Math. \textbf{195} (2021), no.~3, 391--400.
			\href{https://doi.org/10.1007/s00605-020-01505-5}
			{doi:10.1007/s00605-020-01505-5}.
			
			\bibitem{GJPPZ}
			F. Grunewald, A. Jaikin-Zapirain, A. G. S. Pinto, and P. A. Zalesskii,
			\textit{Normal subgroups of profinite groups of non-negative deficiency},
			J. Pure Appl. Algebra \textbf{218} (2014), no.~5, 804--828.
			\href{https://doi.org/10.1016/j.jpaa.2013.10.003}
			{doi:10.1016/j.jpaa.2013.10.003}.
			
			\bibitem{KZ}
			D. H. Kochloukova and P. Zalesskii,
			\textit{Free-by-Demushkin pro-$p$ groups},
			Math. Z. \textbf{249} (2005), no.~4, 731--739.
			\href{https://doi.org/10.1007/s00209-004-0720-6}
			{doi:10.1007/s00209-004-0720-6}.
			
			\bibitem{Labute}
			J. P. Labute,
			\textit{Classification of Demushkin groups},
			Canad. J. Math. \textbf{19} (1967), no.~1, 106--132.
			\href{https://doi.org/10.4153/CJM-1967-007-8}
			{doi:10.4153/CJM-1967-007-8}.
			
			\bibitem{NSW}
			J. Neukirch, A. Schmidt, and K. Wingberg,
			\textit{Cohomology of Number Fields},
			2nd ed.,
			Grundlehren der mathematischen Wissenschaften, vol.~323,
			Springer, Berlin, 2008.
			
			\bibitem{RZ}
			L. Ribes and P. Zalesskii,
			\textit{Profinite Groups},
			2nd ed.,
			Ergebnisse der Mathematik und ihrer Grenzgebiete,
			3. Folge, vol.~40,
			Springer, Berlin--Heidelberg, 2010.
			
			\bibitem{SchmidtWingberg}
			A. Schmidt and K. Wingberg,
			\textit{Extensions of profinite duality groups},
			preprint, 2008,
			arXiv:0809.4371.
			
			\bibitem{Serre}
			J.-P. Serre,
			\textit{Galois Cohomology},
			corrected 2nd printing,
			Springer Monographs in Mathematics,
			Springer, Berlin, 2002.
			
			\bibitem{WingbergPD}
			K. Wingberg,
			\textit{On Poincar\'e groups},
			J. London Math. Soc. (2) \textbf{33} (1986), no.~2, 271--278.
			\href{https://doi.org/10.1112/jlms/s2-33.2.271}
			{doi:10.1112/jlms/s2-33.2.271}.
			
		\end{thebibliography}
\end{document}